\documentclass{amsart}
\usepackage{amsmath, amsthm, amssymb}
\usepackage[all]{xy}
\usepackage{hyperref}
\usepackage{mathrsfs}
\usepackage{researchpreamble}

\numberwithin{equation}{section}

\newcommand{\kkn}{\kummer{K}{n}}

\begin{document}

\title{Period and index of genus one curves over number fields}
\author{Shahed Sharif}
\address{Dept.\ of Mathematics \\
Duke University \\
Durham, NC 27708}
\email{sharif@math.duke.edu}
\thanks{The author would like to thank Bjorn Poonen and Pete Clark for very helpful conversations.}
\date{October 11, 2008}

\begin{abstract}
  The period of a curve is the smallest positive degree of Galois-invariant divisor classes. The index is the smallest positive degree of rational divisors. We construct examples of genus one curves with prescribed period and index over certain number fields.
\end{abstract}

\maketitle

\section{Introduction}
\label{sec:introduction}

Let $X$ be a nonsingular, projective, geometrically integral curve over a field $K$. Define the \emph{index} $I(X)$ to be the greatest common divisor of $[L:K]$, where $L$ varies over finite algebraic field extensions for which $X(L) \neq \emptyset$. In particular, if $X$ has a rational point over $K$ already, the index is 1. Define the \emph{period} $P(X)$ of $X$ to be smallest positive degree amongst Galois-invariant divisor classes on $X$. That is, if $\Kb$ is an algebraic closure for $K$, then we consider divisor classes on $\Xbar := X\times \Kb$ which are fixed points for the Galois action $\Gal(\Kb/K)$. We have $P(X) \mid I(X)$, for the Galois orbit of any $x\in X(L)$ yields an invariant divisor class. However, the two need not be equal: take the example of a conic without rational points, say (the projective curve given by) $x^2 + y^2 = -1$ over $\R$. Clearly, the index is $2$, but the class of a single point is Galois invariant.

In~\cite{lichtenbaum1969}, Lichtenbaum showed that 

\begin{theorem}\label{thm:licht}
  For $X/K$ as above, if the genus, period, and index of $X$ are $g$, $P$, and $I$ respectively, then $P\mid I \mid 2P^2$, $I\mid (2g-2)$, and if either $(2g-2)/I$ or $P$ is even, then $I\mid P^2$.
\end{theorem}

However, over local fields and $C_1$ fields, the above divisibility conditions are not sharp. For example, for genus 1 curves over local fields, the period and index are always equal~\cite{lichtenbaum1969}. One may then ask what triples $(g,P,I)$ actually occur as the genus, period, and index of a curve over a fixed $K$. Over local fields of characteristic not 2, the problem is solved in~\cite{sharif2007}. We consider the case $g=1$ and prove the following:

\begin{theorem}\label{thm:main}
  Let $K$ be a number field and $E$ an elliptic curve over $K$. Let $\ell$, $n$ be positive integers such that $\ell \mid n$. Suppose either that $K$ contains the $n$th roots of unity, or that $E$ has an order $n$ Galois-stable cyclic subgroup. If $n$ is even and $4\nmid \ell$, suppose further that $K$ contains the $(2n)$th roots of unity, or that $E$ has a Galois-stable cyclic subgroup of order $2n$. Then there exists a genus 1 curve $X$ with Jacobian $E$ for which $P(X) = n$ and $I(X) = n \ell$.
\end{theorem}

Note that for $g=1$, $(2g-2)/I$ is always even; therefore by Theorem~\ref{thm:licht} we have $I\mid P^2$. For $K$ satisfying the hypotheses of Theorem~\ref{thm:main}, then, we have covered every possible period and index for genus 1 curves.

The significance of period and index lies in computations of Tate-Shafarevich and Brauer groups. Grothendieck~\cite[\S 4 et seq.]{grothendieck-brauer3} is the canonical source for these computations. Gonzalez-Avil\'es~\cite{gonzalezaviles2003} showed that given a (suitably nice) curve $X$ over a global field $K$, under certain hypotheses the finiteness of the Brauer group of a model for $X$ is equivalent to the finiteness of the Tate-Shafarevich group of the Jacobian of $X$. The relationship between the sizes of these groups is then computable, and depends on the indices and periods of $X$ both over $K$ and over the completions of $K$. Liu, Lorenzini, and Raynaud extended this result (\cite[Theorem 4.3]{liu-lorenzini-raynaud2004} and \cite{liu-lorenzini-raynaud2005}). They showed that for $K$ the function field of a curve over a finite field, and assuming that for some prime $\ell$ the $\ell$-part of $\Br X$ is finite, or that the Tate-Shafarevich group of the Jacobian of $X$ is finite, we have $\Br X$ is finite and has square order.

Lastly, in a forthcoming paper~\cite{clark-sharif} the author with Pete Clark uses a similar result as the one in this paper to show that if $E$ is an elliptic curve over a number field $K$, $p$ is a prime and $N$ any integer, then there exists a $p$-extension $L/K$ such that $\#\Sha(E/L)[p] \geq N$; that is, the order of the $p$-torsion of the Tate-Shafarevich group is unbounded over $p$-extensions.

For other results exhibiting curves with various periods and indices, see Cassels~\cite{cassels1963}, Lang-Tate~\cite{lang-tate1958}, Stein~\cite{stein2002}, O'Neil~\cite{oneil2002}, and Clark~\cite{clark2005}, \cite{clark2006} and \cite{clark2007}. The strongest previous results in this direction are Clark's, who proved that when $E[p] \subset E(K)$ for $p$ prime, there are principal homogeneous spaces for $E$ with period $p$ and index $p^2$, and that there are curves of every index over every number field.

Recently, Jakob Stix~\cite{stix-preprint2008} has shown that if $X$ is a curve over $\Q$ and there is a section for the canonical map $\pi_1(X) \to \pi_1(\Q)$, then the period and index of $X$ are equal.



\section{Preliminary Results}
\label{sec:preliminary-results}

\subsection{Basic properties of period and index}
\label{sec:basic-prop-peri}

One can alternatively define the index to be the smallest positive degree of a divisor on $X$ over $K$. If $x \in X(\Kb)$, the Galois orbit of $x$ written as a divisor $\sum (^\sigma x)$ furnishes a rational divisor of positive degree. Any prime divisor (of $\Div X$, not $\Div \Xbar$) is necessarily of this form. This shows the equivalence of the two definitions. 

There is also an alternate definition of the period. Choose any $x \in X(\Kb)$. The cocycle $\sigma \mapsto\, {^\sigma x - x}$ furnishes a cohomology class in $\xi\in\Hp^1(K, J)$, where $J$ is the Jacobian variety of $X$. Let $\Pic^1_{X/K}$ be the connected component of the Picard scheme of $X$ representing degree 1 invertible sheaves; it is a principal homogeneous space for $J = \Pic^0_{X/K}$. As such, there is a class in $\Hp^1(K, J)$ representing $\Pic^1_{X/K}$, and in fact one sees that this class is $\xi$. Similarly, $n\xi$ represents $\Pic^n_{X/K}$. We know that $\Pic^n_{X/K}$ is a trivial principal homogeneous space for $J$ precisely when it possesses a $K$-point. Such a point occurs when there is a Galois-invariant element of $\Pic^n \Xbar$. Thus we may define the period as the order of $\xi$.

Note that when $X$ is a genus one curve, $X$ itself equals $\Pic^1_{X/K}$, and $\xi$ represents $X$ as an element of $\Hp^1(K, E)$, where $E$ is the Jacobian of $X$. We say a field $L$ \emph{splits} $X$ if $X(L) \neq \emptyset$. For $X$ a genus 1 curve represented by $\xi \in \Hp^1(K, E)$, $L$ splits $X$ if and only if $\xi$ lies in the kernel of the restriction map $\Hp^1(K, E) \to \Hp^1(L, E)$.

We now reduce the proof of Theorem~\ref{thm:main} to the case where $n$ (and hence $\ell$) is a prime power.

\begin{lemma} \label{lem:prime-decomposition}
  Let $P_1, P_2, I_1$, and $I_2$ be positive integers such that the $P_i$ are relatively prime and $P_i \mid I_i \mid P_i^2$. Let $E$ be an elliptic curve over a field $K$. If there are genus one curves $X_1$ and $X_2$ over $K$ with Jacobian $E$ such that $X_i$ has period $P_i$ and index $I_i$, then there is also a curve $X$ with Jacobian $E$ having period $P_1P_2$ and index $I_1I_2$.
\end{lemma}

\begin{proof}
  Let $\xi_i \in \Hp^1(K, E)$ represent $X_i$. I claim that $\xi := \xi_1 + \xi_2$ has period $P_1P_2$ and index $I_1I_2$.

Since the period is the order of $\xi$ in $\Hp^1(K, E)$, the first part of the claim is obvious. Now suppose that $\xi$ has index $I$. If $L$ is a finite extension of $K$ which splits $\xi$, then $\xi$ lies in the kernel of the restriction map $\res: \Hp^1(K, E) \to \Hp^1(L, E)$. Since the orders of the $\xi_i$ are relatively prime and $\res$ is a homomorphism, $L$ splits both of the $\xi_i$ as well. Therefore $I_1I_2$ divides $I$. On the other hand, any field which splits both of the $\xi_i$ splits $\xi$ as well. In particular, we can choose fields of the form $L_1\cdot L_2$ where each $L_i$ splits $\xi_i$. Varying over all such choices, one sees that $I \mid I_1I_2$.
\end{proof}

\subsection{O'Neil's obstruction map}
\label{sec:oneils-obstr-map}

Our main tool for computing the index is O'Neil's obstruction map. In order to define it, we first construct a \emph{theta group}. Let $E[n]$ be the $n$-torsion points of the elliptic curve $E$, viewed as a group scheme over $K$. Then our theta group is a central extension $\sG(n)$ in the category of group schemes over $K$ given by the exact sequence
\begin{equation}
    \label{eq:theta}
0 \to \G_m \to \sG(n) \to E[n] \to 0,
\end{equation}
where the commutator is given by the Weil pairing. That is, $\sG(n)$ as a scheme is just $\G_m \times E[n]$, but with multiplication satisfying
\[
(x, S) (y, T) (x,S)^{-1}(y,T)^{-1}= (e_n(S,T), O),
\]
where $e_n$ is the Weil pairing on $E[n]$ and $O$ is the identity of $E$; see~\cite[\S 23]{mumford1985} or~\cite{mumford1966} for more details. Taking (nonabelian) Galois cohomology, one obtains a map
\[
\Ob: \Hp^1(K, E[n]) \to \Hp^2(K, \G_m) = \Br K.
\]
This is the obstruction map. Note that it is a \emph{quadratic} map, not a homomorphism. (Quadratic essentially means that $\Ob(\xi) = b(\xi,\xi)$ for some bilinear map $b$.) This results from the fact that $\sG(n)$ is a noncommutative group scheme, and hence our cohomology sequence takes place in nonabelian Galois cohomology; see Zarhin~\cite{zarhin1974}.

Recall the Kummer sequence for $E$ gives rise to the exact sequence in Galois cohomology
\begin{equation}
  \label{eq:kummer-E}
0 \to E(K)/nE(K) \to \Hp^1(K, E[n]) \to \Hp^1(K, E)[n] \to 0.
\end{equation}
Thus, if $X$ has period dividing $n$, there exist (usually many) $\xi \in \Hp^1(K, E[n])$ representing it; that is, the image of $\xi$ in $\Hp^1(K, E)[n]$ represents $X$ in the usual sense.

In~\cite{oneil2002}, O'Neil showed
\begin{proposition}\label{prop:ob-index}
  The image of the obstruction map lies in $(\Br K)[n]$. If $\Ob(\xi)$ has order $\ell$, then the index of the curve $X$ represented by $\xi$ divides $n\ell$. If $X$ has index $n$, then there is a class $\xi$ representing $X$ such that $\Ob(\xi) = 0$.
\end{proposition}

The idea is that $\Hp^1(K, E[n])$ classifies principal homogeneous spaces $X$ for $E$ over $K$ \emph{along with} a choice of Galois-invariant degree $n$ invertible sheaf $\sL \in (\Pic \Xbar)^G$, up to isomorphism over $K$. Specifically, any cocycle in the class of $\xi$ gives rise to a $\Kb$-isomorphism $\phi:E \to X$. Then $\sL = \sL(n\phi(O))$. One sees that $\Gamma(\Xbar, \sL)/\Kb^\times$ forms a twist of $\Pro^{n-1}$ over $K$---that is, a Brauer-Severi variety. The obstruction map takes the pair $(X, \sL)$ to the class of this Brauer-Severi variety in $\Br K$. From this characterization, the proposition is easy: If $X$ has index $n$, then there is a rational divisor $D$ of degree $n$. The class $\xi$ representing $(X, \sL(D))$ satisfies $\Ob(\xi) = 0$. If $\Ob(\xi)$ has order $\ell$ and $\xi$ represents the pair $(X, \sL)$, then the Brauer-Severi variety arising from $\sL$ is split by an extension of degree $\ell$. Then the tensor product $\sL^\ell$ is a class of degree $n\ell$ containing a rational divisor. See~\cite{oneil2002} for more details.

One would hope that the converse held: if $X$ has index $n\ell$, then there ought to be $\xi$ representing $X$ such that $\Ob(\xi)$ has order $\ell$. However, this is not known to be true. We will use a trick in proving the main theorem to construct $X$ for which this does hold.


Let $\delta$ be the composition $E(K) \to E(K)/nE(K) \to \Hp^1(K, E[n])$, where the latter map is given by the map in~\eqref{eq:kummer-E}. For $x \in E(K)$ and $X$ a principal homogeneous space for $E$ over $K$, let $T(x, X)$ denote the Tate pairing of $X$ with the image of $x$ in $E(K)/nE(K)$. Then

\begin{proposition} \label{prop:ob-tate}
  $\Ob (\xi + \delta x) = \Ob(\xi) + T(x,X)$.
\end{proposition}

\begin{proof}
  See~\cite[\S 5]{oneil2002}.
\end{proof}

O'Neil's introduction of the obstruction map is useful because in many cases it may be computed using a Hilbert symbol, which we now define. Let $K$ be a field of characteristic not dividing $n$ which contains $\mu_n$, the $n$th roots of unity.  Let $a, b \in K^\times/K^{\times n}$. By Kummer theory, we know that $K^\times/K^{\times n} = \Hp^1(K, \mu_n)$. The cup product gives a map
\[
\kkn \to \Hp^2(K, \mu_n \otimes \mu_n).
\]
Fix a primitive $n$th root of unity $\zeta$. Define an isomorphism $\mu_n \otimes \mu_n \to \mu_n$ by $\zeta^i \otimes \zeta^j \mapsto \zeta^{ij}$, which induces an isomorphism $\Hp^2(K, \mu_n\otimes\mu_n) \to \Hp^2(K, \mu_n)$. Using the fact that $\Hp^2(K, \mu_n) = (\Br K)[n]$, we see that the composition gives a pairing 
\begin{align*}
\kkn & \to (\Br K)[n] \\
(a,b) & \mapsto \langle a,b\rangle
\end{align*}
which we call the Hilbert symbol. If we let $w = (a,b)$, then we will alternatively write $\langle w\rangle$ for $\langle a,b\rangle$. Since the cup product is bilinear and skew-symmetric, so is the Hilbert symbol. 

Note that the Hilbert symbol depends on $n$ and $\zeta$. Frequently we will abuse notation and define the Hilbert symbol as a map $(K^\times)^2 \to \Br K$, by composing with the obvious quotient map.

We return to the obstruction map. Assume that the $n$-torsion of $E$ is rational over $K$. By the theory of the Weil pairing, $\mu_n \subset K$. Let $\zeta$ be the previously chosen primitive $n$th root of unity. Fix a basis $(S,T)$ for $E[n]$ such that $e(S, T) = \zeta$, where $e$ is the Weil pairing.. The choice of basis, and our fixed generator $\zeta$ of $\mu_n$, yields an isomorphism of (trivial) Galois-modules $E[n] \isom \mu_n \times \mu_n$, and we have an isomorphism
\[
\kappa: \Hp^1(K, E[n]) \to \kkn.
\]

\begin{proposition}\label{prop:ob-hilb}
    Let $\xi \in \Hp^1(K, E[n])$. If $n$ is odd or $E[2n] \subset E(K)$, then $\Ob(\xi) = \langle \kappa(\xi)\rangle$. If $n$ is even, then $2\Ob(\xi) = 2\langle \kappa(\xi) \rangle$.
\end{proposition}

\begin{proof}
    For the case that $n$ is odd, see~\cite[Prop.~3.4]{oneil2002} and \cite{oneil2004}. A proof in the even case, including an explicit computation of $\Ob(\xi) - \langle \kappa(\xi) \rangle$, can be found in~\cite[\S 2]{clark-sharif}.
\end{proof}

Assume from now on that $K$ is a number field. If $v$ is a place of $K$, write $K_v$ for the completion of $K$ at $v$. In order to use the Hilbert symbol, we will reduce to the local case using the fact that, if $K$ is a global field, $\Br K = \oplus_v \Br K_v$. That is, in order to compute $\langle a, b\rangle$, it suffices to compute $\langle a, b\rangle_v$, where the latter symbol is computed in $K_v$, and $a,b$ are considered as elements of $K_v$.

(More generally, for any $K$-group scheme $M$ and integer $q$, we have a natural localization map
\[
\Hp^q(K, M) \to \Hp^q(K_v, M)
\]
in \'etale cohomology. We will denote this map by adding the subscript $v$; e.g.\ $\xi \mapsto \xi_v$.)

If $v$ is a place of $K$, we also use $v$ to denote a fixed corresponding valuation.
\begin{lemma}\label{lem:hilb-local-fields}
    Let $K_v$ be a nonarchimedean local field such that $v(n) = 0$. Let $\pi$ be a uniformizing parameter and $u, u'$ units; that is, $v(u) = v(u') = 0$. Let $\F$ be the residue field of $K_v$.
    \begin{enumerate}
        \item $\langle u, u'\rangle_v = 0$.
        \item The order of $\langle u, \pi \rangle_v$ equals the order of the image of $u$ in $\F^\times/\F^{\times n}$.
    \end{enumerate}
\end{lemma}

\begin{proof}
  According to~\cite[Ch.\ XIV]{serrelocalfields}, $\langle a, b\rangle_v = 0$ if and only if $b$ is a norm from the extension $K_v(a^{1/n})/K_v$. The extension $K_v(u^{1/n})$ is unramified with degree, say, $d$, which by Hensel's Lemma also equals the order of $u$ in $\F^\times/\F^{\times n}$. According to local class field theory, the norm from an unramified extension of degree $d$ is precisely the set of elements $x$ such that $v(x)$ is a multiple of $dv(\pi)$. The result follows immediately for the first claim, and from bilinearity for the second claim.
\end{proof}

We also have the following:
\begin{proposition} \label{prop:product-formula}
    For $a,b\in K^\times$, $\displaystyle\sum_v \langle a, b\rangle_v = 0$.
\end{proposition}
Here, we use the fact that $\Br K_v$ is canonically isomorphic to $\Q/\Z$, $\frac12 \Z/\Z$, or $0$ via the \emph{invariant map} so that the sum makes sense. For the remainder of this paper, we identify $\Br K_v$ with the appropriate subgroup of $\Q/\Z$. Note that the sum is finite: there are finitely many archimedean primes, and for all but finitely many nonarchimedean $v$, $v(a) = v(b) = v(n) = 0$, so that by Lemma~\ref{lem:hilb-local-fields}, $\langle a,b\rangle_v = 0$. 

\begin{proof}
    This follows from the fact that $\Br K$ is isomorphic to the set of $(t_v)$, $t_v \in \Br K_v$, such that almost all $t_v =0$ and $\sum t_v = 0$; and the fact that cup products, and hence the Hilbert symbol, commute with localization.
\end{proof}

In order to deal with the fact that the obstruction map does not necessarily give a sharp bound on index, we will use obstruction maps of different levels. We define $\Ob_n$ to be the corresponding obstruction map $\Hp^1(K, E[n]) \to \Br K$.

\begin{proposition}\label{prop:ob-different-levels}
    Let $n$ and $m$ be positive integers. The following diagrams commute:
    \begin{enumerate}
        \item 
\[
\xymatrix{\Hp^1(K, E[n]) \ar[rr]^{\Ob_n} \ar[d]_{j_*} & & \Br K \ar[d]^{m}\\
  \Hp^1(K, E[mn]) \ar[rr]^{\Ob_{mn}} & & \Br K
}
\]
where $j_*$ is induced by the canonical inclusion $j: E[n] \to E[mn]$, and $m$ is muliplication by $m$.
        \item
\[
\xymatrix{\Hp^1(K, E[mn]) \ar[rr]^{\Ob_{mn}} \ar[d]_{[m]} & & \Br K \ar[d]^{m}\\
  \Hp^1(K, E[n]) \ar[rr]^{\Ob_{n}} & & \Br K
}
\]
where $[m]$ is the map induced by the multiplication by $m$ map $E[mn] \to E[n]$.
        \item 
\[
\xymatrix{\Hp^1(K, E[n]) \ar[rr]^{\Ob_n} \ar[d]_{\res} & & \Br K \ar[d]^{\res}\\
  \Hp^1(L, E[n]) \ar[rr]^{\Ob_{n}} & & \Br L
}
\]
where $L/K$ is a field extension and $\res$ is the restriction map.
\end{enumerate}
\end{proposition}

\begin{proof}
    Let $\xi \in \Hp^1(K, E[n])$ represent the pair $(X, \sL)$, where $\sL$ is the divisor class of $n(x)$ for some $x \in X(\Kb)$. Then $j_*(\xi)$ represents $(X, \sL^m)$. The map which takes elements of $(\Pic \Xbar)^G$ to the class of the associated Brauer-Severi variety is a homomorphism; indeed, it arises from the Leray spectral sequence $\Hp^p(K, \Hp^q(\Xbar, \G_m)) \Rightarrow \Hp^{p+q}(X, \G_m)$, which yields the exact sequence
\[
0 \to \Pic X \to \Hp^0(K, \Pic \Xbar) \to \Hp^2(K, \G_m).
\]
The last map sends an invertible sheaf to the class of its associated Brauer-Severi variety; see for example~\cite{lichtenbaum1968}. The first part of the proposition follows.

For the second part, Mumford showed~\cite[p.\ 309--310]{mumford1966} that multiplication by $m$ extends to a homomorphism on theta groups $\sG(mn) \to \sG(n)$. The restriction of this homomorphism to $\G_m$ is also multiplication by $m$. Taking the long cohomology sequence associated to~\eqref{eq:theta}, we obtain the commutative diagram.

The third part is obvious, since $\res \xi$ represents the same pair $(X, \sL)$ as $\xi$ does, but over $L$. 
\end{proof}

Similar to $\Ob_n$, we define $\delta_n$ to be the composition $E(K) \to E(K)/nE(K) \to \Hp^1(K, E[n])$, where the second map in the composition is the coboundary coming from the Kummer sequence for $E$.

\section{The case of rational torsion}
\label{sec:case-rati-tors}

By Lemma~\ref{lem:prime-decomposition}, we may assume that $n$ is a prime power $p^r$; this assumption holds for the remainder of the paper. In this section, we will prove the theorem under the additional assumption that $E[n] \subset E(K)$ in the case that $p$ is odd, and $E[2n] \subset E(K)$ if $p=2$. 

\subsection{Choosing a pair of primes}
\label{sec:choosing-pair-primes}

Our first step is to find a pair of distinct nonarchimedean primes $v$, $v'$ satisfying certain splitting conditions. (We use primes and places interchangeably.) We state the conditions below, after which we show that infinitely many pairs $v$, $v'$ satisfying the conditions exist. The proof is essentially repeated use of the Cebotarev density theorem. Let $S$ be the union of the primes $w$ of $K$ such that $E$ has bad reduction at $w$, archimedean primes, and primes dividing $n$. The conditions are

\begin{itemize}
    \item[A1.] The primes $v$,$v'$ are principal with totally positive generators $\pi$ and $\pi'$ respectively.
    \item[A2.] Let $E(K)$ embed in $E(K_v)$ in the usual manner. Then $E(K)$ lies in $nE(K_v)$.
    \item[A3.] For each $w\in S$, $\pi$ and $\pi'$ lie in $K_w^{\times n}$.
    \item[A4.] The order of the image of $\pi'$ in $K_v^\times/K_v^{\times n}$ is exactly $n$.
\end{itemize}

\begin{lemma} \label{lem:rational-torsion-pairs-primes}
    There exist infinitely many pairs of distinct primes $v$, $v'$ satisfying conditions A1--A4.
\end{lemma}

\begin{proof}
    Condition A1 is equivalent to $v$ and $v'$ splitting completely in the Hilbert class field of $K$.

Condition A2 is the same as requiring $v$ to split completely in $K([n]^{-1}E(K))$; that is, the field obtained by adjoining to $K$ all $x\in E(\Kb)$ such that $[n]x\in E(K)$. By~\cite[p.194]{silvermanAEC}, $K([n]^{-1}E(K))$ is a finite abelian extension which is unramified outside $S$.

Let $\sm$ be the modulus given by the product of $n^2$ and all primes where $E$ has bad reduction. Let $K_\sm$ be the ray class field for $K$ with modulus $\sm$. Then condition A3 holds if $v$ and $v'$ split completely in $K_\sm$; for in that case, the Frobenius for $v$ (say) is trivial and, by class field theory, $v$ has a generator $\pi$ which is congruent to $1\pmod{\sm}$. Then Hensel's Lemma implies A3.

We now choose $v$ to be any prime which splits completely in the compositum of the three fields named above. Next we tackle A4.

By abuse of notation, let $v$ be a valuation corresponding to the prime $v$. Let $\alpha\in K$ be any element such that $v(\alpha) = 0$ and whose image in $K_v^\times/K_v^{\times n}$ has order $n$; since $\pi \cong 1 \pmod{\sm}$, and $n \mid \sm$, there exists such $\alpha$ in $K_v$. But $K$ is dense in $K_v$, so we may find such $\alpha$ in $K$. Let $F'$ be the ray class field with modulus $v$. By class field theory, the Galois group $\Gal(F'/K)$ is isomorphic to the class group with modulus $v$. In particular, if $v'$ and $(\alpha)$ lie in the same class in this class group, then $v'$ has a generator $\pi'$ which is congruent to $\alpha \pmod{v}$, and hence satisfies A4.

Let $F$ be the compositum of $K_\sm$ and $K([n]^{-1}E(K))$. We see that $F'$ is unramified outside $v$, while $F$ is unramified at $v$. Hence $F \cap F'$ lies in the Hilbert class field of $K$. To satisfy A1--A3, we wish $v'$ to split completely in $F$, while to satisfy A4, we wish $v'$ to lie in the class of $(\alpha)$ in the appropriate ray class group. These conditions are compatible since they both imply that $v'$ splits completely in the Hilbert class field of $K$. Thus we may apply the Cebotarev density theorem to find infinitely many such $v'$ and $\pi'$.
\end{proof}

\subsection{Construction of curve}
\label{sec:ratl-torsion-cocycles}

As mentioned earlier, a choice of ordered basis $(S,T)$ for $E[n]$ with $e(S,T) = \zeta$ yields an isomorphism
\[
\kappa:\Hp^1(K, E[n]) \to \kummer{K}{n}.
\]
Choose $\xi\in \Hp^1(K, E[n])$ such that $\kappa(\xi) = (\pi, \pi'^{n/\ell})$. Let $X$ be the corresponding principal homogeneous space for $E$.

\begin{proposition} \label{prop:end-ratl-torsion}
    The curve $X$ has period $n$ and index $n\ell$.
\end{proposition}

\begin{proof}
    For any positive integer $m$, we have $\kappa(m\xi) = (\pi^m, \pi'^{mn/\ell})$. Let $v$ be the place of $K$ lying over $(\pi)$, and suppose $m < n$. If the curve associated to $m\xi$ is a trivial principal homogeneous space, then there exists $x\in E(K)$ such that $\kappa (\delta x + m\xi) = (1,1)$, where we recall that $\delta$ is the composition
\[
E(K) \to E(K)/nE(K) \to \Hp^1(K, E[n]).
\]
But by condition A2, $x$ lies in $nE(K_v)$. Hence $(\delta x)_v = 0$. Therefore for any choice of $x$, $\kappa (\delta x + m\xi)_v = (\pi^m,\, \cdot \,)$, and so $m\xi$ yields a trivial principal homogeneous space if and only if $n \mid m$. Therefore the period of $X$ is $n$.

By Proposition~\ref{prop:ob-hilb}, $\Ob(\xi) = \langle \pi, \pi'^{n/\ell}\rangle$. We  compute the Hilbert symbol locally. For places $w$ satisfying $w(n) > 0$, condition A3 shows that 
\[\langle \pi,\pi'^{n/\ell}\rangle_w = 0.\]
Let $v'$ be the place corresponding to $\pi'$. For $w \neq v, v'$ and $w(n) = 0$, $\pi$ and $\pi'$ are both units in $K_w$, and so by Lemma~\ref{lem:hilb-local-fields} the local Hilbert symbol is zero. By Proposition~\ref{prop:product-formula}, the order of the Hilbert symbol at $v$ equals that at $v'$, so we need only consider $v$. Combining condition~A4 with Lemma~\ref{lem:hilb-local-fields}, we see that $\langle \pi, \pi'^{n/\ell}\rangle$ has order $\ell$. Hence $\Ob(\xi)$ has order $\ell$, and the index of $X$ divides $n\ell$.

Let $\ell'$ strictly divide $\ell$, and let $j_*$ be the canonical map
\[
\Hp^1(K, E[n]) \to \Hp^1(K, E[n\ell']).
\]
If $X$ had index $n\ell'$, then by Lemma~\ref{prop:ob-index} there would exist $x$ in $E(K)$ such that 
\[\Ob_{n\ell'}(\delta_{n\ell'} x + j_*(\xi)) = 0.
\]
According to Proposition~\ref{prop:ob-tate}, the above equals
\[
\Ob_{n\ell'} (j_*(\xi)) + T(x, X).
\]
By condition~A2, $x\in nE(K_v)$ and hence $T(x,X)_v = 0$. But by Proposition~\ref{prop:ob-different-levels}, $\Ob_{n\ell'} j_*(\xi) = \ell' \Ob_n \xi \neq 0$ if $\ell' <\ell$. Therefore the index of $X$ is precisely $n\ell$.
\end{proof}

\section{Proof of Main Theorem, odd $n$}
\label{sec:fields-cont-roots}

For now, we assume $n$ is odd. We suppose either: a) $\mu_n \subset K$, or b) $E$ contains a Galois-stable subgroup of order $n$. Over $L:=K(E[n])$, we may use the arguments of the previous section to construct a cohomology class $\xi\in\Hp^1(L, E[n])$ with the desired properties. Let $\cores$ be the corestriction map
\[
\cores:\Hp^1(L, E[n]) \to \Hp^1(K, E[n]).
\]
We will show that $\cores \xi$ represents a curve over $K$ with period $n$ and index $n\ell$. In order to compute the index, we will need to base extend back to $L$; in other words, we'll compute $\res \circ \cores \xi$, where $\res$ is the restriction map
\[
\res:\Hp^1(K, E[n]) \to \Hp^1(L, E[n]).
\]

\subsection{Pairs of primes}
\label{sec:pairs-general}

As in the previous section, we wish to choose principal primes of $L$, $(\pi)$ and $(\pi')$, such that analogues to conditions A1--A4 hold, along with a new condition. For any place $w$ of $L$, let $w_K$ denote the place of $K$ lying below $w$. Let $S$ be the set of primes $w$ of $L$ such that $E$ has bad reduction at $w_K$, $w_K$ is archimedean, or $w_K(n)>0$.

We now state the relevant conditions.
\begin{itemize}
    \item[B1.] The primes $v = (\pi)$ and $v' = (\pi')$ are principal, with totally positive generators $\pi$ and $\pi'$. 
    \item[B2.] Under the usual embedding, $E(K)$ lies in $nE(K_{v_K})$.
    \item[B3.] The generators $\pi$ and $\pi'$ lie in $L_w^{\times n}$ for all $w \in S$.
    \item[B4.] The order of the image of $\pi'$ in $L_{v}^{\times}/L_{v}^{\times n}$ is $n$. Additionally, $\sigma \pi'$ lies in $L_{v}^{\times n}$ for all nontrivial $\sigma \in \Gal(L/K)$.
    \item[B5.] The primes $v_K, v_K'$ split completely in $L$.
\end{itemize}

\begin{lemma} \label{lem:cyclotomic-pairs-primes}
  There are infinitely many pairs of primes $(\pi)$, $(\pi')$ satisfying conditions B1--B5.
\end{lemma}

\begin{proof}
  Let $\sm$ be the modulus over $L$ given as the product of $n^2$ and all primes $w$ in $S$. Let $L_\sm$ be the ray class field of $L$ with modulus $\sm$. The modulus $\sm$ is rational over $K$, so that $L_\sm$ is Galois over $K$. Let $F$ be the compositum of $L_\sm$, $K([n]^{-1}E(K))$, and the Hilbert class field of $K$. By the Cebotarev density theorem, there exists infinitely many primes $v_K$ of $K$ which split completely in $F$. Setting $v$ equal to any prime of $L$ lying over $v_K$, we use the same reasoning as in Lemma~\ref{lem:rational-torsion-pairs-primes} to see that $v$ satisfies conditions B1--B5.

Choose any unit $\beta$ in $L_v$ which has order $n$ in $L_v^\times/L_v^{\times n}$. By the Chinese Remainder Theorem, there exists $\alpha \in L$ such that
\begin{align} \label{eq:alpha-cong}
  \alpha & \equiv \beta \pmod{v} \notag \\
  \alpha & \equiv 1 \pmod{\sigma v} \quad \forall \sigma\neq 1 \in \Gal(L/K)
\end{align}

Let $F'$ be the ray class field for $L$ with modulus $\prod \sigma v$. The modulus is rational over $K$, so that $F'$ is Galois over $K$. The Galois group $\Gal(F'/L)$ is isomorphic to the class group with modulus $\prod \sigma v$ via the Artin reciprocity map. Let $\gamma_L$ in $\Gal(F'/L)$ map to the class of $(\alpha)$ under this isomorphism. Then under the inclusion $\Gal(F'/L) \hookrightarrow \Gal(F'/K)$, $\gamma_L$ maps to, say, $\gamma$. Let $[\gamma]$ be the conjugacy class of $\gamma$ in $\Gal(F'/K)$. 

One can find $\tau \in \Gal(F'F/K)$ such that $\tau|_{F'}$ lies in $[\gamma]$, and $\tau|_{F}$ is trivial. This follows because $F'/L$ is ramified only at the $\sigma v$, while $F/L$ is unramified at those places; therefore $F\cap F'$ is contained in the Hilbert class field of $L$. But $\gamma$ acts trivially on the Hilbert class field since $(\alpha)$ is principal. Therefore, $\tau$ as prescribed exists.

Now we apply the Cebotarev density theorem again to find $v'_K$ corresponding to the conjugacy class of $\tau$.
Choosing $v'$ to be any place of $L$ lying over $v'_K$, we use the same reasoning as before to conclude that $v'$ satisfies the conditions.
\end{proof}

\subsection{The corestriction map}
\label{sec:cores}

As before, choose a basis $(S,T)$ for $E[n]$ such that $e(S, T) = \zeta$. If we are under the hypothesis that $E$ contains a Galois-stable order $n$ subgroup, choose our basis so that $S$ generates this subgroup. As before, our choice of basis yields an isomorphism
\[
\kappa: \Hp^1(L, E[n]) \to \kummer{L}{n} 
\]
Let $\pi, \pi'$ be as in the previous section, and choose $\xi$ so that $\kappa(\xi) = (\pi,\pi')$. Let $\cores:\Hp^1(L, E[n]) \to \Hp^1(K, E[n])$ be the corestriction map, and $\eta = \cores(\xi)$. We wish to show that the curve corresponding to $\eta$ satisfies the conditions of our theorem. But it is too difficult to compute the obstruction map $\Ob(\eta)$, so we will instead use $\Ob(\res \eta)$, where $\res$ is the restriction map
\[
\Hp^1(K, E[n]) \to \Hp^1(L, E[n]).
\]
Therefore, we need to compute $\res \circ \cores (\xi)$. Note that $\cores \circ \res$ is well-known and equal to $[L:K]$, but the same is not true of $\res\circ\cores$. 

For any $G_K$-module $M$ and nonnegative integer $r$, one has an action of $\Gal(L/K)$ on $\Hp^r(L, M)$. This action is induced by the action on homogeneous $r$-cochains
\[
c^\sigma (\gamma_1, \dots, \gamma_r) = \sigma c(\sigma^{-1}\gamma_1\sigma, \dots, \sigma^{-1}\gamma_r\sigma)
\]
where $\sigma \in \Gal(L/K)$ and $\gamma_i \in G_L$. (Actually, we must lift $\sigma$ to any fixed element of $G_K$ when acting by conjugation on the $\gamma_i$.) Define $\Nm: \Hp^r(L,M) \to \Hp^r(L,M)$ by
\[
\Nm(\theta) = \sum_{\sigma\in\Gal(L/K)} \theta^\sigma.
\]
\begin{lemma}
  For $\theta \in \Hp^r(L, M)$, $\res\circ\cores \theta = \Nm \theta$, where $\res$ and $\cores$ are the obvious restriction and corestriction maps.
\end{lemma}

\begin{proof}
  This follows from the definition of corestriction and dimension shifting; see~\cite[p. 119]{serrelocalfields}.
\end{proof}

Therefore $\cores \eta = \Nm \xi$. Unfortunately, $\kappa$ is not $\Gal(L/K)$-equivariant, so that $\kappa(\Nm \xi) \neq (\Nm \pi,\Nm \pi')$. Instead, we have a \emph{twisted} norm on $\kummer{L}{n}$, which we now describe. Let
\begin{align*}
  \label{eq:1}
  \Gal(L/K) &\to \Gl_2(\Z/n\Z) \\
  \sigma &\mapsto M_\sigma
\end{align*}
be the representation of $\Gal(L/K)$ on $E[n]$ with respect to the basis $(S,T)$. The determinant of $M_\sigma$ is given by the $n$th cyclotomic character. If $\mu_n \subset K$, then the determinant is identically 1, and hence the representation lies in $\Sl_2(\Z/n\Z)$. If $S$ generates a Galois-stable subgroup, then the representation lies in the set of upper-triangular matrices.

\begin{proposition}\label{prop:G-action-cohomology}
  Suppose $\sigma \in \Gal(L/K)$, $\xi \in \Hp^1(L, E[n])$ and $\kappa(\xi) = (a,b)$. Then
\[
\kappa(\xi^\sigma) = \frac{M_\sigma}{\det M_\sigma} (\sigma a,\sigma b)
\]
where for $M\in \Gl_2(\Z/n\Z)$, if 
\[
M = \left[
  \begin{array}{cc}
    i & j \\
    k & l 
  \end{array}\right],
\]
then $M (a,b) = (a^ib^j, a^kb^l)$.
\end{proposition}

\begin{proof}
  Let $\rho: E[n] \to \mu_n\times\mu_n$ be the isomorphism given by the basis $(S,T)$. Write $(\mu_n\times\mu_n)_\rho$ for the $\Gal(L/K)$-module with underlying group $\mu_n\times\mu_n$, but with module-structure making $\rho$ an isomorphism of Galois modules; that is
  \begin{align*}
\sigma(\zeta_1,\zeta_2)_\rho &= \rho\sigma\rho^{-1}(\zeta_1,\zeta_2)\\
&= M_\sigma(\zeta_1, \zeta_2).
  \end{align*}
Consider the diagram
\begin{equation} \label{mu_n rho}
\xymatrix{\kummer{L}{n} \ar[r]^-\psi & \Hp^1(L,\mu_n\times \mu_n) \ar[d]^{i_*} & \\
& \Hp^1(L,(\mu_n\times\mu_n)_\rho) \ar[r]_-{\phi} & \Hp^1(L,E[n])}.
\end{equation}
The isomorphism $\psi$ comes from the usual Kummer isomorphism. The map $i_*$ is induced by the canonical isomorphism of the underlying groups $i:\mu_n \times \mu_n \to (\mu_n\times\mu_n)_\rho$. The map $\phi$ is induced by the map sending $(\zeta, 1)$ to $S$ and $(1,\zeta)$ to $T$ (compare to $\kappa^{-1}$). Then the horizontal arrows are both $\Gal(L/K)$-equivariant, and the composition $\phi i_*\psi$ is equal to $\kappa^{-1}$. 

Let $\gamma \in \Gal(\bar{L}/L)$ and lift $\sigma$ arbitrarily to an element, also written $\sigma$, of $\Gal(\Kb/K)$. We have
\begin{align}
  [i_*\psi(a,b)]^\sigma(\gamma) &= \sigma i(\psi(a,b)(\sigma^{-1}\gamma\sigma)) \\
  &= \sigma i(\sigma^{-1}\sigma \psi(a,b)(\sigma^{-1}\gamma\sigma)) \\
  &= \sigma i(\sigma^{-1}\psi(a,b)^\sigma(\gamma)) \label{eq:2} \\
  &= \sigma i(\sigma^{-1}[\psi(\sigma a, \sigma b)(\gamma)]) \label{eq:3}\\
  &= \frac{M_\sigma}{\det M_\sigma} i[\psi(\sigma a, \sigma b)(\gamma)] \label{eq:4}\\
  &= i_*\psi\left(\frac{M_\sigma}{\det M_\sigma} (\sigma a, \sigma b)\right) (\gamma).
\end{align}
The equality~(\ref{eq:2}) follows by the definition of the $\Gal(L/K)$-action on $\Hp^1(L, \mu_n\times\mu_n)$. Since $\psi$ is a $\Gal(L/K)$-morphism, we obtain~(\ref{eq:3}). Finally, (\ref{eq:4}) follows from the action of $\sigma$ on $E[n]$ by the matrix $M_\sigma$, and the action of $\sigma^{-1}$ on $\mu_n\times \mu_n$ by $(\det M_\sigma)^{-1}$.

Now apply $\phi$ to both sides of the equation to obtain
\[
[\kappa^{-1}(a,b)]^\sigma = \kappa^{-1}\left(\frac{M_\sigma}{\det M_\sigma} (\sigma a, \sigma b)\right)
\]
from which the lemma follows.
\end{proof}

\begin{corollary}\label{cor:explicit-norm}
  Let $\kappa(\xi) = (a,b)$. Then $\kappa(\Nm \xi)$ equals
    \[\left(\prod (\det M_\sigma)^{-1}\sigma a^{i(\sigma)} \sigma b^{j(\sigma)},\, \prod (\det M_\sigma)^{-1} \sigma a^{k(\sigma)} \sigma b^{\ell(\sigma)}\right)\]
where the product is taken over $\sigma \in \Gal(L/K)$.
\end{corollary}

\subsection{Computation of the obstruction map}
\label{sec:ob-cores}

Let $\xi \in \Hp^1(L, E[n])$ satisfy $\kappa(\xi) = (\pi,\pi'^{n/\ell})$, where $\pi$ and $\pi'$ were chosen as in Lemma~\ref{lem:cyclotomic-pairs-primes}. Let $\eta = \cores \xi$. In order to compute $\Ob (\eta)$, we instead compute $\res \Ob (\eta) = \Ob (\res \eta)$, or $\Ob (\Nm \xi)$. But the $n$-torsion $E[n]$ is rational over $L$, so by Proposition~\ref{prop:ob-hilb} we may use the Hilbert symbol to compute $\Ob (\Nm \xi)$. Let
\begin{align*}
  c &= \prod (\det M_\sigma)^{-1} \sigma \pi^{i(\sigma)} \\
  d & = \prod (\det M_\sigma)^{-1} \sigma \pi^{k(\sigma)} \\
  c' &= \prod (\det M_\sigma)^{-1} \sigma \pi'^{\frac{n}{\ell}j(\sigma)} \\
  d' & = \prod (\det M_\sigma)^{-1} \sigma \pi'^{\frac{n}{\ell}l(\sigma)}.
\end{align*}
Thus, $\Nm \kappa^{-1}(\pi, 1) = \kappa^{-1}(c,d)$, $\Nm \kappa^{-1}(1,\pi'^{n/\ell}) = \kappa^{-1}(c',d')$, and $\kappa(\Nm \xi) = (cc',dd')$. We wish to compute the Hilbert symbol $\langle cc', dd'\rangle$.

\begin{lemma}\label{lem:cyc-local-triviality}
  Let $w$ be any place of $L$ satisfying $w(\pi\pi') = 0$. Then the local Hilbert symbol $\langle cc',dd'\rangle_w$ is trivial
\end{lemma}

\begin{proof}
  If $w$ is non-archimedean and $w(n)=0$, then $cc'$ and $dd'$ are both units in $L_w$. Therefore by Lemma~\ref{lem:hilb-local-fields} the Hilbert symbol is trivial.

If $w(n)\neq 0$, then $w$ lies in $S$. According to condition B3, $\pi$, $\pi'$ and their conjugates lie in $L_w^{\times n}$. Again, the Hilbert symbol is trivial.

Lastly, since $n$ is an odd prime power and $\mu_n \subset L$, $L$ has no real embeddings. Therefore the Hilbert symbol at all the archimedean places is automatically trivial.
\end{proof}

\begin{lemma} \label{lem:cyc-hilb-v}
  Let $v$ be the place of $L$ corresponding to $\pi$. 
Then $\langle cc',dd'\rangle_v$ has exact order $\ell$. 
\end{lemma}

\begin{proof}
    By bilinearity, we have
\[
\langle cc',dd'\rangle = \langle c,d\rangle+ \langle c,d'\rangle+\langle c',d\rangle +\langle d,d'\rangle.
\]
Since $c',d,d'$ are all units in $L_v^\times$, the last two terms are zero by Lemma~\ref{lem:hilb-local-fields}. Now $c = u\pi$, where $u$ is some unit in $L_v^\times$, so that $\langle c,d'\rangle_v = \langle \pi, d'\rangle_v$. By condition B4, $d' \equiv \pi'^{n/\ell} \pmod{L^{\times n}}$, so that $\langle c,d'\rangle_v = \langle \pi, \pi'^{n/\ell}\rangle_v$. Again applying condition B4 and Lemma~\ref{lem:hilb-local-fields}, we see that $\langle c,d'\rangle_v$ has order $\ell$. 

We now show that $\langle c,d\rangle_v = 0$. One sees that it suffices to show that $\langle \pi, d\rangle_v = 0$.

Recall that $(S,T)$ was our chosen basis for $E[n]$. If the subgroup generated by $S$ is Galois-stable over $K$, then the $M_\sigma$ are all upper triangular. In particular, $k(\sigma) = 0$ for all $\sigma \in \Gal(L/K)$, which implies that $d = 1$. Therefore $\langle \pi, d\rangle_v = 0$.

Let us assume instead that $\mu_n \subset K$. This immediately implies that $\det M_\sigma = 1$. Since $n$ is odd, it suffices to show that $2\langle \pi, d\rangle_v = 0$. By expanding out $d$ and using bilinearity, we may instead show that
\[
\langle \pi, \sigma \pi^{k(\sigma)}\rangle = -\langle \pi, \sigma^{-1} \pi^{k(\sigma^{-1})}\rangle_v.
\]
Recall that we write $M_\sigma$ by
\[
M_\sigma = \left[
  \begin{array}{cc}
    i(\sigma) & j(\sigma) \\
    k(\sigma) & l(\sigma) 
  \end{array}\right].
\]
But $M_{\sigma^{-1}} = M_{\sigma}^{-1}$. Using $\det M_\sigma = 1$, we obtain
\[
M_\sigma^{-1} = \left[
  \begin{array}{cc}
    l(\sigma) & -j(\sigma) \\
    -k(\sigma) & i(\sigma) 
  \end{array}\right].
\]
In other words, $k(\sigma^{-1}) = -k(\sigma)$. Therefore it suffices to show that $\langle \pi, \sigma \pi\rangle_v = \langle \pi, \sigma^{-1}\pi\rangle_v$. But we have
\begin{align*}
  \langle \pi, \sigma \pi\rangle_v &= \langle \sigma^{-1}\pi, \pi\rangle_{\sigma^{-1}v} \\
  &= -\langle \pi, \sigma^{-1}\pi\rangle_{\sigma^{-1}v} \\
  &= \langle \pi, \sigma^{-1}\pi\rangle_v.
\end{align*}
The last equality follows from Proposition~\ref{prop:product-formula} and the fact that the symbol $\langle \pi, \sigma^{-1}\pi\rangle$ is trivial everywhere but at $v$ and $\sigma^{-1}v$. This proves the lemma.
\end{proof}

Let $v'$ be the place of $L$ corresponding to $\pi'$.
\begin{proposition} \label{prop:ob-eta}
  $\Ob \eta$ has order $\ell$ at $v_K$ and $v'_K$, and is trivial at all other places.
\end{proposition}

\begin{proof}
  First we prove triviality. For any place $w_K$ of $K$, the corestriction map on $\Hp^1(L, E[n])$ induces a homomorphism
\[
\oplus\cores_w: \oplus \Hp^1(L_w,E[n]) \to \Hp^1(K_{w_K}, E[n])
\]
where the sum is over all places $w$ of $L$ lying over $w_K$. Recall that $\eta = \cores \xi$. By construction, $\xi$ is locally trivial everywhere except at $v$, $v'$, and their conjugates. Triviality at $w_K \neq v_K, v'_K$ follows.

Since the local invariants of a global Brauer class sum to zero, it suffices to show that $(\Ob \eta)_{v_K}$ has order $\ell$. The diagram
\[
\xymatrix{\Hp^1(K, E[n]) \ar[d] \ar[r] & \Br K \ar[d] \\
  \Hp^1(K_{v_K}, E[n]) \ar[d]^{\oplus \res} \ar[r] & \Br K_{v_K} \ar[d]^{\oplus \res} \\
  \oplus \Hp^1(L_v, E[n]) \ar[r] & \oplus \Br L_v
}
\]
commutes. The horizontal maps are the relevant obstruction maps. Commutativity follows from functoriality of localization and restriction in nonabelian cohomology. But $v_K$ splits completely in $L$ by condition~B5, so that $K_{v_K} \isom L_v$. Thus if we consider the restriction map onto a single factor
\[
\xymatrix{\Hp^1(K_{v_K}, E[n]) \ar[d] \ar[r] & \Br K_{v_K} \ar[d] \\
  \Hp^1(L_v, E[n]) \ar[r] &\Br L_v}
\]
then the vertical maps are isomorphism. Therefore the order of $(\Ob \eta)_{v_K}$ equals the order of $(\Ob \res \eta)_v$, or rather $(\Ob \Nm \xi)_v$. 

Since $E[n] \subset E(L) \subset E(L_v)$, we may use the Hilbert symbol to evaluate the obstruction map. By construction, $\kappa(\Nm \xi) = (cc',dd')$. The result follows from Lemma~\ref{lem:cyc-hilb-v}.
\end{proof}

\subsection{End of proof}
\label{sec:end-of-proof}

Let $X$ be the genus 1 curve over $K$ represented by the class $\eta$ in $\Hp^1(K, E[n])$. Clearly the period of $X$ divides $n$. Suppose that the period is smaller. Then there is some positive integer $m$ with $m<n$ and some $x\in E(K)$ such that
\[
\delta x + m\eta = 0.
\]
In particular, this holds locally at $v_K$. But by condition~B2, $\delta x$ is trivial at $v_K$ for all $x\in E(K)$. Therefore we must have $m\eta = 0$ at $v_K$. In particular, this must hold when we restrict to $L_v$. But this cannot be, for 
\begin{align*}
\kappa(\res_{L/K} m\eta) &= m\kappa(\res \eta) \\
&= m(u_1\cdot \pi, u_2) \\
& = (u_1^m\cdot \pi^m, u_2^m)
\end{align*}
for some $u_1$, $u_2$ which are units in $L_v$. We conclude that $X$ has exact period $n$.

By Proposition~\ref{prop:ob-eta}, $\Ob \eta$ has order $\ell$. Applying Proposition~\ref{prop:ob-index}, we see that the index of $X$ divides $n\ell$. Suppose the index is $n\ell'$, with $\ell'<\ell$. By Proposition~\ref{prop:ob-index} there is a class $\eta'$ in $\Hp^1(K, E[n\ell'])$ representing $X$ such that $\Ob_{n\ell'} \eta' = 0$. Let $j_*$ be the canonical homomorphism
\[
\Hp^1(K, E[n]) \to \Hp^1(K, E[n\ell']).
\]
Since $\eta$ also represents $X$, there exists $x\in E(K)$ such that
\[
\eta' = j_*(\eta) + \delta x.
\]
By Proposition~\ref{prop:ob-tate},
\[
\Ob_{n\ell'} \eta' = \Ob_{n\ell'} j_*(\eta) + T(x,X)
\]
where $T$ is the Tate pairing. Consider the latter equation locally at $v_K$. By condition~B2, $x$ lies in $nE(K_{v_K})$, so the Tate pairing is zero. The left hand term is zero by hypothesis. But by Proposition~\ref{prop:ob-different-levels}, $\Ob_{n\ell'} j_*(\eta) = \ell' \Ob_n \eta$, and the latter is \emph{not} zero since $\Ob_n \eta$ has order $\ell$ and $\ell'<\ell$. This yields a contradiction. Therefore the index of $X$ must be exactly $n\ell$.

\section{Even period}
\label{sec:even-period}

Assume now that $n$ is a power of 2. Three problems now arise in the previous arguments.

The first problem is that the obstruction map need not equal the Hilbert symbol; instead, according to Proposition~\ref{prop:ob-hilb}, $\Ob(\xi) - \langle \kappa(\xi)\rangle$ is killed by $2$. The second problem is that in the proof of Lemma~\ref{lem:cyc-local-triviality}, $\langle cc',dd'\rangle_w$ need not be trivial at real $w$; but in any case $2\langle cc', dd'\rangle_w = 0$. Finally, in our proof of Lemma~\ref{lem:cyc-hilb-v}, we showed that $\langle \pi, d\rangle_v = 0$ when $\mu_n \subset K$ by instead showing that $2\langle \pi, d\rangle_v = 0$; if $n$ is even, of course, this is insufficient.

Suppose we undertake our construction anyway, yielding $\eta \in \Hp^1(K, E[n])$. The three problems above imply that our calculation of $\Ob (\eta)$ may be off by an element of $(\Br K)[2]$. If $\ell \geq 4$, then the difference is immaterial, and the proof works. We now consider the cases $\ell = 1$ and $\ell = 2$.

Assume that $\ell = 1$. According to our main theorem, we are under the weakened hypothesis that either $\mu_{2n} \subset K$ or $E$ contains a Galois-stable cyclic subgroup of order $2n$. Using the construction from the previous section, we can come up with a curve $X'$ with period $2n$ and index either $2n$ or $4n$, though we are not able to determine which of the two holds. Suppose $\eta\in \Hp^1(K, E[2n])$ represents $X'$ with $\Ob_{2n} (\eta) \in (\Br K)[2]$. Let $X$ be the curve with class $2\eta$. Note that we may view $2\eta$ as an element of $\Hp^1(K, E[n])$; in particular, $X$ has period $n$. To show that $X$ has index $n$, it suffices to show that $\Ob_n (2\eta) = 0$. But by Proposition~\ref{prop:ob-different-levels}, $\Ob_n (2\eta) = 2\Ob_{2n} \eta = 0$.


The procedure for $\ell = 2$ is similar: construct $X'$ as before with period $2n$ and index $8n$, so that $\Ob_{2n} (\eta)$ has order $4$. Then $\Ob_{n} (2\eta)$ has order $2$, and the curve $X$ represented by $2\eta$ has period $n$ and index $2n$.

\section{Final remarks}
\label{sec:final-remarks}

The hypothesis on the main theorem amounts to a statement about the Galois representation on $E[n]$. The hypothesis is needed at precisely one place: the calculation of $\langle \pi, d\rangle_v$ appearing in the proof of Lemma~\ref{lem:cyc-hilb-v}. Recall that $d$ is a product of the conjugates of $\pi$, each appearing with some multiplicity. If one could, for example, choose $\pi$ so that $\langle \pi, \sigma \pi\rangle = 0$ for all $\sigma$, then one could generalize the theorem to arbitrary number fields. 

One can get around this to some extent, at least when we wish the index to be maximal ($I = P^2$); see~\cite{clark-sharif}, which deals with this case over arbitrary number fields.

The results of the paper should generalize to function fields, that is, finite extensions of $\F_p(T)$. The only wrinkle occurs when $p \mid n$, for then $E[n]$ is no longer an \'etale group scheme. Clark, in a personal communication, has results making use of a so-called flat Hilbert symbol. One suspects that, with greater care, the results of this paper can be duplicated using the flat symbol.

Lastly, the description of the $\Gal(L/K)$ action on $\Hp^1(L, E[n])$ yields an explicit description, in most cases, of $\Hp^1(K, E[n])$ for $K$ an arbitrary number field, in the following manner. Given $E/K$, for $n$ not divisible by a finite set of primes, we know that $\Gal(L/K)$ surjects onto $\Aut(E[n])$. One can show that $\Hp^q(L/K, E[n]) = 0$ for all $q$, whence the inflation-restriction sequence yields an isomorphism 
\[
\Hp^1(K, E[n]) = \Hp^1(L, E[n])^{\Gal(L/K)}.
\]
From Proposition~\ref{prop:G-action-cohomology}, one obtains the desired description of $\Hp^1(K, E[n])$ as pairs of elements of $L^\times/L^{\times n}$.

\bibliography{/home/postdoc/sharif/Research/biblio.bib}
\bibliographystyle{amsplain}
\end{document}